\documentclass[draft]{amsart}
\usepackage{amsfonts}
\usepackage{amssymb}
\usepackage{latexsym}
\usepackage{color}
\usepackage{euscript}  
\usepackage{mathrsfs}
\numberwithin{equation}{section} \theoremstyle{plain}
\usepackage{units}
\usepackage{setspace}

\makeatletter
\@namedef{subjclassname@2010}{%
\textup{2010} Mathematics Subject Classification}
\makeatother

\newtheorem{thm}{Theorem}[section]
\newtheorem{prop}[thm]{Proposition}
\newtheorem{cor}[thm]{Corollary}
\newtheorem{lem}[thm]{Lemma}

\newtheorem*{hpt*}{Hipoteza}
\newtheorem*{prob*}{Problem}
\newtheorem*{thm*}{Theorem}
\newtheorem*{pro*}{Proposition}
\newtheorem*{met*}{Method}
\newtheorem*{lem*}{Lemma}

\DeclareMathOperator{\D}{d\!} \DeclareMathOperator{\E}{e}
 \DeclareMathOperator{\dom}{D}
 \DeclareMathOperator{\M}{m}
\DeclareMathOperator{\esup}{\textrm{ess sup }}

\theoremstyle{definition}

\newtheorem{exa}[thm]{{\it Example}}
\newtheorem*{exa*}{{\it Example}}

\theoremstyle{remark}
\newtheorem{rem}[thm]{{\it Remark}}
\newtheorem*{rem*}{{\it Remark}}

\def\C{\mathbb{C}}
\def\R{\mathbb{R}}

\def\Z{\mathbb{Z}}

\def\A{A}

\def\ff{\mathcal F}
\def\aa{\mathcal{A}}

\def\bb{\mathcal{B}}

\def\bb{\mathcal{B}}

\def\hh{\mathcal H}

\def\N{\mathbb N}

\def\munfty{\mu_{\mathtt{G}}}
\newcommand*{\MUN}[1]{\mu_{\mathtt{G},#1}}

\newcommand{\Le}{\leqslant}
\newcommand{\Ge}{\geqslant}
\newcommand*{\bor}[1]{\mathfrak B(#1)}
\newcommand*{\LIM}{\mathtt{lim}\,}
\newcommand*{\ILIM}{\mathtt{LIM}\,}
\newcommand*{\psnmu}{(X_n,\aa_n,\mu_n)}
\newcommand*{\psnnu}{(X_n,\aa_n,\nu_n)}
\newcommand*{\lin}{\mathtt{lin}\,}

\newcommand\frx{\mathscr{X}}
\newcommand\frs{\mathscr{F}}
\newcommand\frm{\mathscr{M}}
\newcommand\frc{\mathscr{C}}
\newcommand\frcn{\mathscr{C}_n}
\newcommand\frck{\mathscr{C}_k}

\newcommand*{\esf}{\mathsf{E}}
\newcommand*{\hsf}{\mathsf{h}}

\begin{document}

\title[Composition operators via inductive limits]{dense definiteness and boundedness of composition operators in $L^2$-spaces via inductive limits}
\author[P. Budzy\'{n}ski]{Piotr Budzy\'{n}ski}
\address{Katedra Zastosowa\'n Matematyki, Uniwersytet Rolniczy w Krakowie, ul. Balicka 253c, 30-198 Krak\'ow, Poland}
\email{piotr.budzynski@ur.krakow.pl}

\author[A. P{\l}aneta]{Artur P{\l}aneta}
\address{Katedra Zastosowa\'n Matematyki, Uniwersytet Rolniczy w Krakowie, ul. Balicka 253c, 30-198 Krak\'ow, Poland}
\email{artur.planeta@ur.krakow.pl}
\begin{abstract}
The question of dense definiteness and boundedness of composition operators in $L^2$-spaces are studied by means of inductive limits of operators. Methods based on projective systems of measure spaces and inductive limits of $L^2$-spaces are developed. Illustrative examples are presented.
   \end{abstract}
\maketitle
\section{Introduction}
Bounded composition operators (in $L^2$-spaces) have been extensively studied since the works of Koopman and von Neumann (see \cite{koo,koo-neu}). They played a central role in ergodic theory and proved to be important objects of investigations in operator theory. Many properties of these operators were fully characterized (see the monograph \cite{sin-man} and references therein). Unbounded composition operators attracted attention recently but they turned out to have very interesting attributes (cf. \cite{bud-dym-pla,bud-jab-jun-sto1,bud-jab-jun-sto2,bud-jab-jun-sto3,cam-hor,jab, kum-kum}). In particular, they proved to be a source for surprising examples (cf. \cite{bud,bud-dym-jab-sto,bud-jab-jun-sto5,jab-jun-sto}).

In this paper we investigate the questions of dense definiteness and boundedness of composition operators. These properties have characterizations (cf. \cite{bud-jab-jun-sto4,cam-hor,nor}), which in a more concrete situations seem difficult to apply. For example, this is the case of a composition operator induced by an infinite matrix in $L^2(\munfty)$, where $\munfty$ is the gaussian measure on $\R^\infty$. Even in bounded case in a concrete situations the question of boundedness may be highly non-trivial and may lead to very interesting results (cf. \cite{dan-sto,mla,sto,sto-sto}). We show that a technique based on inductive limits might be helpful when dealing with these problems. We deliver tractable criteria for the above mentioned properties. This is possible if the $L^2$-space (in which a given composition operator acts) is an inductive limit of $L^2$-spaces with underlying measure spaces forming a projective system (see Section \ref{projsyst}). In this case we prove that both the dense definiteness and boundedness can be expressed in terms of asymptotic behaviour of appropriate Radon-Nikodym derivatives (see Theorems \ref{absconA} and \ref{sufgen}). We illustrate this with examples.
\section{Preliminaries}
\subsection{Notation}
In all what follows $\Z$ stands for the set of integers and $\N$ for the set of positive integers; $\R$ denotes the set of real numbers, $\C$ denotes the set of complex numbers. If $X$ is any subset of $\R$, then by $X_+$ we understand the set $\{x\in X\colon x\Ge 0\}$. Set $\overline{\R}_+=\R_+\cup\{\infty\}$. By $\sigma_1\triangle\sigma_2$ we denote the symmetric difference $(\sigma_1\cup\sigma_2)\setminus(\sigma_1\cap\sigma_2)$ between sets $\sigma_1$ and $\sigma_2$. For a topological space $X$, $\bor X$ stands for the family of Borel subsets of $X$. If $\{X_n\}_{n\in\N}$ is a sequence of subsets of a set $X$, then ``$X_n\nearrow X$ as $n\to\infty$''means that $X_n\subseteq X_{n+1}$ for all $n\in\N$ and $X=\bigcup_{n\in\N} X_n$.

Let $\hh$ be a (complex) Hilbert space and $T$ be an operator in $\hh$ (all operators are assumed to be linear in this paper). By $\dom(T)$ we
denote the domain of $T$. $\overline{T}$ stands for the closure of $T$. $\bb(\hh)$ denotes the Banach space of all bounded operators
on $\hh$ (with usual supremum norm). If $T$ is closable and $\ff$ is a subspace of $\hh$ such that $\overline{T|_\ff}=\overline{T}$, then $\ff$ is said to be  a core of $T$.

Let $(X,\aa,\mu)$ be a measure space. By $(\aa)_\mu$ we denote the collection of all $\sigma\in\aa$ such that $\mu(\sigma)<\infty$. 
Let $1\Le p<\infty$. The space of all $\aa$-measurable complex-valued functions such that $\int|f|^p\D\mu<\infty$ is denoted by $L^p(\mu)=L^p(X,\aa,\mu)$; $L^\infty(\mu)=L^\infty(X,\aa,\mu)$ stands for the space of all complex-valued and $\mu$-essentially bounded functions on $X$.

Now, let $\{\mu_{n}\}_{n\in\N}$ be a sequence of non-negative measures, each $\mu_n$ acting on a measurable space $(X_n,\aa_n)$. Let $\{f_n\colon n\in\N\}$ be a family of functions such that $f_n\in L^1(\mu_n)$ for every $n\in\N$. Then $\frm\big(\{f_n\colon n\in\N\}\big)$ stands for the family composed of monotonically increasing convex functions $G:[0,\infty)\to[0,\infty)$ such that $\lim_{t\to\infty}G(t)/t=\infty$ and $\sup_{n\in\N}\int G(|f_n|)d\mu_n<\infty$.
\subsection{Composition operators}
Let $(X,\aa,\mu)$ be a $\sigma$-finite measure space and let $A\colon X\to X$ be an $\aa$-measurable. Define the measure $\mu\circ A^{-1}$ on $\aa$ by setting $\mu\circ A^{-1} (\sigma)=\mu(A^{-1}(\sigma))$, $\sigma\in\aa$. If $\mu\circ A^{-1}$ is absolutely continuous with respect to $\mu$, then $A$ is said to be nonsingular transformation of $X$.  If $A$ is nonsingular, then the linear operator
\begin{align*}
C_A\colon L^2(\mu) \supseteq \dom(C_A) \to L^2(\mu)
\end{align*}
given by
   \begin{align*}
\dom(C_A)=\{f \in L^2(\mu) \colon f\circ A \in L^2(\mu)\} \text{ and } C_A f=f\circ A \text{ for } f\in \dom(C_\phi),
   \end{align*}
is well-defined and closed in $L^2(\mu)$ (cf. \cite[Proposition 1.5]{bud-jab-jun-sto1}). Such an operator is the {\it composition operator induced by $A$} and $A$ is the {\em symbol} of $C_A$. Usually, properties of $C_A$ are written in terms of the Radon-Nikodym derivative
\begin{align*}
\hsf^A=\frac{\D \mu\circ A^{-1}}{\D\mu}.
\end{align*}
By the measure transport theorem (\cite[Theorem 1.6.12]{ash-dol}) for every $\aa$-measurable ($\C$- or $\overline{\R}$-valued) function $g$ we have
\begin{align*}
\int_X g\cdot \hsf^A\D\mu=\int_X g\circ A\D\mu.
\end{align*}
In particular, for every $f\in\dom(C_A)$ there is
\begin{align}\label{RN}
\int_X |f|^2\hsf^A\D\mu=\int_X |f\circ A|^2\D\mu.
\end{align}
It is known (cf.\ \cite[Proposition 4.2]{bud-jab-jun-sto1}) that
\begin{align}\label{dense}
\text{$\overline{\dom(C_A)}=L^2(\mu)$ if and only if $\hsf^A<\infty$ a.e.\ $[\mu]$}.
\end{align}
If $\hsf^A$ belongs to $L^\infty(\mu)$, then $C_A$ is bounded on $L^2 (\mu)$ (and {\em vice versa}) and
\begin{align}\label{norm}
\|C_A\|=\|\hsf^A\|_{L^\infty(\mu)}^{1/2}.
\end{align}
Conditional expectation is indispensable when investigating composition operators in $L^2$-spaces. We collect here some of its properties.  Set $A^{-1}(\aa)=\{A^{-1}(\sigma)\colon \sigma \in \aa\}$. Assume that $\hsf^A < \infty$ a.e.\ $[\mu]$. Then the measure $\mu|_{A^{-1}(\aa)}$ is $\sigma$-finite (cf.\ \cite[Proposition 3.2]{bud-jab-jun-sto1}) and hence for every $\aa$-measurable function  $f\colon X \to \overline{\R}_+$ there exists a unique (up to sets of $\mu$-measure zero) $A^{-1}(\aa)$-measurable function\footnote{For simplicity we do not make the dependence of $\esf(f)$ on $A$ explicit} $\esf(f)\colon X \to \overline{\R}_+$ such that for every
$\aa$-measurable function $g\colon X \to \overline{\R}_+$,
    \begin{align} \label{CE-3}
    \int_X g \circ A \cdot f \D\mu= \int_X g \circ A \cdot \esf(f) \D\mu.
    \end{align}
We call $\esf(f)$ the {\em conditional expectation} of $f$ with respect to $A^{-1}(\aa)$ (see \cite{bud-jab-jun-sto1, bud-jab-jun-sto2, bud-jab-jun-sto3} for more information on $\esf(\cdot)$ in the context of unbounded composition operators and further references). It is known that if $f\colon X \to \overline{\R}_+$ is $\aa$-measurable function, then $\esf(f) = g\circ A$ a.e.\ $[\mu]$, where $g\colon X \to \overline{\R}_+$ is an $\aa$-measurable function such that $g=0$ a.e.\ $[\mu]$ on $\{\hsf^A=0\}$. Set $\esf(f) \circ A^{-1} = g$ a.e.\ $[\mu]$. This definition is correct (see \cite{cam-hor} and \cite[Appendix B]{bud-jab-jun-sto1}).
Moreover, we have
   \begin{align} \label{fifi}
    (\esf(f) \circ A^{-1})\circ A = \esf(f) \quad \text{a.e.\ $[\mu|_{A^{-1}(\aa)}]$.}
   \end{align}
It is also known that the map $f\mapsto \esf(f)$ can be extended linearly from $\{f\in L^2(\mu)\colon f\Ge 0\}$ onto the whole $L^2(\mu)$ in a way that $\esf(\cdot)$ becomes an orthogonal projection acting on $L^2(\mu)$. This (extended) conditional expectation $\esf(\cdot)$ satisfies \eqref{CE-3} and \eqref{fifi} with $f, g\in L^2(\mu$).

Let $\mu$ be the Borel measure on $\R^n$, $n\in\N$,  given by $\D\mu=\rho \D \M_n$, where  $\rho\colon\R^n\to(0,\infty)$ is a Borel function and $\M_n$ is the $n$-dimensional Lebesgue measure on $\R^n$. If $A$ is an invertible linear transformation of $\R^n$, then by the measure transport theorem we have
\begin{align*}
\hsf^{A}=\tfrac{1}{|\det A|} \tfrac{\rho\circ A^{-1}}{\rho};
\end{align*}
moreover, if $C_A$ is bounded on $L^2(\mu)$, then
\begin{align*}
\|C_A\|^2=\tfrac{1}{|\det A|} \big\|\tfrac{\rho\circ A^{-1}}{\rho}\big\|_{L^\infty(\mu)}.
\end{align*}
In particular, if $\rho(x_1,\ldots,x_n)=\E^{-\frac12(x_1^2+\cdots+x_n^2)}$, we have
    \begin{align}\label{norminv}
    \hsf^{A}(x)=\big|\det A\big|^{-1} \exp \tfrac12 \Big( \|x\|^2-\|A^{-1}(x)\|^2\Big)\quad \text{for $\M_n$-a.e. $x\in\R^n$}.
    \end{align}
\subsection{Inductive limits}
Suppose $\{\hh_n\}_{n\in\N}$ is a sequence of Hilbert spaces. We say that a Hilbert space $\hh$ is the {\it inductive limit} of $\{\hh_n\}_{n\in\N}$
if there are isometries $\Lambda_k^l:\hh_k\rightarrow\hh_l$, $k\leqslant l$, and $\Lambda_k:\hh_k\rightarrow\hh$ such that the following conditions are satisfied:
\begin{enumerate}
 \item[($\mathtt{I}_1$)] $\Lambda_k^k$ is the identity operator on $\hh_k$,
 \item[($\mathtt{I}_2$)] $\Lambda_k^m=\Lambda_l^m\circ\Lambda_k^l$ for all $k\leqslant l\leqslant m$,
 \item[($\mathtt{I}_3$)] $\Lambda_k=\Lambda_l\circ\Lambda_k^l$ for all $k\leqslant l$,
 \item[($\mathtt{I}_4$)] $\hh=\overline{\bigcup_{n\in \N}\Lambda_n\hh_n}$.
\end{enumerate}
We write $\hh=\ILIM\hh_n$ then.

Assume that $\hh=\ILIM\hh_n$. For $n\in\N$, let $C_n$ be an operator in $\hh_n$. Consider the subspace $\dom_\infty=\dom_\infty(\{C_n\}_{n\in \N})$ of $\hh$ given by
\begin{align*}
\dom_\infty=\bigcup_{k\in\N}\{\Lambda_k f\ |\ \exists M \Ge k \colon \Lambda_k^m f\in \dom(C_m) \text{ for all } m\Ge M\}
\end{align*}
and define the operator $\LIM C_n$ in $\hh$ by
    \begin{align*}
    &\dom(\LIM C_n)=\{\Lambda_kf\in \dom_\infty\colon\lim_{m\rightarrow\infty} \Lambda_m C_m \Lambda_k^m f \text{ exists}\}\\
    &(\LIM C_n) \Lambda_kf=\lim_{m\rightarrow\infty} \Lambda_m C_m\Lambda_k^m f,\quad \Lambda_kf\in \dom(\LIM C_n).
\end{align*}
We call $\LIM C_n$ the {\em inductive limit} of $\{C_n\}_{n\in\N}$.

The following lemma is surely folklore. We include the proof for completeness.
\begin{lem}\label{indbound}
Let $\hh=\ILIM\hh_n$ and let $C_n\in\bb(\hh_n)$, for $n\in\N$. Assume that the operator $\LIM C_n$ is densely defined in $\hh$. Then the following assertions hold.
\begin{enumerate}
\item If $\sup_{n\in\N}\|C_n|_{\Lambda_l^n\hh_l}\|<\infty$ for all $l\in\N$, then
$\bigcup_{m=1}^\infty \Lambda_m\hh_m\subseteq\dom (\LIM C_n)$ and $(\LIM C_n)|_{\Lambda_k\hh_k}$ is bounded for all $k\in\N$.  
\item If $\sup_{n\in\N}\|C_n\|<\infty$, then $\LIM C_n$ is closable and $\overline{\LIM C_n}\in\bb(\hh)$.
\end{enumerate}
\end{lem}
\begin{proof}
Fix $k\in\N$ and chose $f\in\hh_k$. Since $\overline{\dom\big(\LIM C_n\big)}=\hh$, for any $\varepsilon>0$ there exists $l\in\N$, $g\in\hh_l$ and $N\in\N$
such that $\|\Lambda_k f-\Lambda_lg\|\Le \varepsilon$ and $\|\Lambda_m C_m\Lambda_k^m g-\Lambda_{m'} C_{m'}\Lambda_k^{m'} g\|\Le \varepsilon$ for every $m,m'\Ge N$. We may assume that $l\Ge k$. Then for all $m,m'\Ge \max\{N,l\}$ we have
    \begin{align*}
    \|\Lambda_m C_m\Lambda_k^m f-\Lambda_{m'} C_{m'}\Lambda_k^{m'} f\|
        &\Le
    \|\Lambda_m C_m\Lambda_k^m f-\Lambda_m C_m\Lambda_l^m g\|\\
        &+
    \|\Lambda_m C_m\Lambda_l^m g-\Lambda_{m'} C_{m'}\Lambda_l^{m'} g\|\\
        &+
    \|\Lambda_{m'} C_{m'}\Lambda_l^{m'} g-\Lambda_{m'} C_{m'}\Lambda_k^{m'} f\|\\
        &\Le
    \varepsilon \Big(1+2\sup_{n\in\N}\|C_n|_{\Lambda_k^n\hh_k}\| \Big).
    \end{align*}
This implies that $\{\Lambda_m C_m \Lambda_k^m f\}_{m=1}^\infty$ is a Cauchy sequence and thus, by definition,  $\Lambda_kf\in\dom(\lim C_n)$. Since $f$ can be chosen arbitrarily we get the inclusion $\Lambda_k\hh_k\subseteq\dom (\LIM C_n)$. The fact that $(\LIM C_n)|_{\Lambda_k\hh_k}$ is bounded follows immediately from $\sup_{n\in\N}\|C_n|_{\Lambda_k^n\hh_k}\|<\infty$, $k\in\N$, and definition of $\LIM C_n$.

By (1) we have $\bigcup_{n=1}^\infty \Lambda_n\hh_n=\dom(\LIM C_n)$ and $\|\LIM C_n f\|\Le\sup_{n\in\N}\|C_n\| \|f\|$ for all $f\in\bigcup_{n=1}^\infty \Lambda_n\hh_n$. This and $\hh=\overline{\bigcup_{n\in \N}\Lambda_n\hh_n}$ implies the claim of (2).
\end{proof}
Regarding Lemma \ref{indbound} it is worth noticing that condition $\sup_{n\in\N}\|C_n|_{\Lambda_l^n\hh_l}\|<\infty$, $l\in\N$, is not sufficient for $\overline{\LIM C_n}\in\bb(\hh)$. This is shown in the following example.
\begin{exa}
Let $\hh_n=L^2([\frac{1}{n},n])$ for $n\in\N$ and $\hh=L^2((0,\infty))$, where, for a given sub-interval $P$ of $\R$, $L^2(P)$ denotes the Hilbert space of all complex functions on $P$, which are square-integrable with respect to Lebesgue measure $\M_1$ on $\R$ (with a standard inner product). Put
\begin{align*}
(\Lambda_kf)(x)=\left\{%
\begin{array}{ll}
    f(x), & \hbox{for $x\in[\frac{1}{k},k]$;} \\
    0, & \hbox{for $x\in(0,\frac{1}{k})\cup(k,\infty)$.}
\end{array}%
\right.
\end{align*}
and $\Lambda_k^n f=\Lambda_k f|_{[\frac{1}{n},n]}$, where $f\in\hh_k$.
It is easily seen that $\hh=\ILIM \hh_n$. Define $\phi(x)=\frac{1}{x}$ for $x>0$. Let $\phi_n=\phi|_{[\frac{1}{n},n]}$ be a restriction of $\phi$ to $[\frac{1}{n},n]$.  Then by the change-of-variable theorem (cf. \cite[Theorem 7.26]{rud}) we have
\begin{align*}
\| C_{\phi_n}\Lambda_k^n f\|^2&=\int_{[\frac{1}{k},k]}\big|f\big(\tfrac{1}{t}\big)\big|^2\D\M(t)\\
&= \int_{[\frac{1}{k},k]} \frac{|f(t)|^2}{t^2}\D \M(t)\Le k^2\| f\|^2,\quad f\in\hh_k, \quad k\in\N.
\end{align*}
This implies that $\sup_{n\in\N}\|C_{\phi_n}|_{\Lambda_l^n\hh_l}\|\Le l$ for $l\in\N$. Observe that $\overline{\LIM C_{\phi_n}}= C_\phi$. Indeed, since $C_\phi\circ\Lambda_n=\Lambda_n\circ C_{\phi_n}$ we see that $\LIM C_{\phi_n}\subseteq C_\phi$. On the other hand, if $f\in\dom(C_\phi)$, then $\Lambda_mf_m \to f$ and $\LIM C_{\phi_n}(\Lambda_mf_m )=C_\phi(\Lambda_mf_m)\to C_\phi f$ as $m\to\infty$, where $f_n=f|_{[\frac{1}{n},n]}$ for $n\in\N$. This proves that $\overline{\LIM C_{\phi_n}}= C_\phi$. Since $\|C_\phi\chi_{[\frac{1}{n},1]}\|=n-1$ for every $n\in\N$, we see that $\overline{\LIM C_{\phi_n}}\notin \bb(\hh)$.
\end{exa}

Below we give an example of a sequence $\{X_n\}_{n\in\N}$ of sets and a sequence $\{\phi_n\}_{n\in\N}$ of transformations
such that the inductive limit $\LIM C_{\phi_n}$ of composition operators is not densely defined.
\begin{exa}
For $n\in\N$, let $X_n=\{1,2,\ldots,n\}$ and let $\mu_n$ be the atomic measure on $X_n$ given by $\mu_n(\{k\})=1$, $k\in X_n$.
It is evident that $\ell^2(\N)=\ILIM L^2(\mu_n)$, where $\ell^2(\N)$ denotes the Hilbert space
of all square-summable complex sequences enumerated by natural numbers (with the standard inner product).
For $n\in\N$, we define a transformation $\phi_n$ of $X_n$ by $\phi(1)=2$, $\phi(2)=3$, $\ldots$ , $\phi(n)=1$.
Obviously, $C_{\phi_n}$ is a bounded operator on $\hh_n$ for every $n\in\N$. In fact, it is unitary one.
However, $C_{\phi_n}\chi_{\{1\}}=\chi_{\{n\}}$ for all $n\in\N$, which implies that $\chi_{\{1\}}$ does not belong to $\dom(\lim C_n)$.
In particular, this means that $\lim C_n$ is not densely defined.
\end{exa}
\begin{rem}
Questions whether $\LIM C_n$ is densely defined and closable are delicate ones. Marchenko type conditions (see \cite{mar}, also \cite{jan}), implying positive answers to both of them, seem difficult to apply in the context of composition operators.
\end{rem}
\section{Projective systems of measure spaces}\label{projsyst}
In this section we study inductive limits of composition operators over $\sigma$-finite measure spaces endowed with projective structure.

Suppose that $\{\psnmu\}_{n\in\N}$ is  a sequence of (not necessarily $\sigma$-finite) measure spaces. If there exist surjective mappings
\begin{equation*}
\delta_m^n:X_m\rightarrow X_n,\quad n\leqslant m,
\end{equation*}
satisfying the following conditions
\begin{enumerate}
 \item[($\mathtt{P}_1$)] $\delta_m^n$ is $(\aa_m,\aa_n)$-measurable for all $n\leqslant m$,
 \item[($\mathtt{P}_2$)] $\delta_k^n=\delta_m^n\circ\delta_k^m$ for all $n\leqslant m\leqslant k$,
 \item[($\mathtt{P}_3$)] $\delta_n^n$ is the identity mapping on $X_n$ for all $n$,
\end{enumerate}
then $\{\psnmu\}_{n\in\N}$ is called a {\em projective system}. We say that a measure space $(X,\aa,\mu)$ is a {\em target space} of the projective system $\{\psnmu\}_{n\in\N}$ if there are surjective mappings
\begin{equation*}
\delta^n:X\rightarrow X_n,\quad n\in\N,
\end{equation*}
that satisfy the following conditions
\begin{enumerate}
    \item[($\mathtt{P}_4$)] $\delta^n$ is $(\aa,\aa_n)$-measurable for all $n$,
    \item[($\mathtt{P}_5$)] $\delta^n=\delta_m^n\circ\delta^m$ for all $n\leqslant m$,
    \item[($\mathtt{P}_6$)] $\aa=\sigma(\{(\delta^n)^{-1}(\sigma)\ :\ \sigma\in\aa_n,\ n\in\N\})
        $\footnote{if $\bb$ is any family of subsets of a set $X$, then by $\sigma(\bb)$ we denote
        the smallest $\sigma$-algebra in $X$ containing family $\bb$.},
    \item[($\mathtt{P}_7$)] $\mu((\delta^n)^{-1}(\sigma))=\lim\limits_{m\rightarrow\infty}\mu_m((\delta_m^n)^{-1}(\sigma))$\footnote{The sets of the form $\big(\delta_k\big)^{-1}(\sigma)$ for $\sigma\in\aa_k$ and $k\in\N$ are called {\em cylinder} sets.}
        for all $\sigma\in\aa_n$ and $n$.
\end{enumerate}
If this is the case, then we write $(X,\aa,\mu)=\LIM\psnmu$ and call $\mu$ a {\it target measure} of $\{\mu_n\}_{n\in\N}$.

Suppose $(X,\aa,\mu)=\LIM\psnmu$. Let $k\in\N$. It is clear that if $(X_k,\aa_k,\mu_k)$ is $\sigma$-finite and
\begin{align}\label{18.05.2011.01}
\mu\circ(\delta^k)^{-1}\ll \mu_k\ \text{ and }\ \frac{\D\mu\circ(\delta^k)^{-1}}{\D\mu_k}\in L^{\infty}(\mu_k),
\end{align}
then the operator
\begin{align*}
\Delta_k:L^2(\mu_k)\ni f\mapsto f\circ\delta^k\in L^2(\mu),
\end{align*}
is well-defined and bounded. Now, if for every $k\in\N$, $(X_k,\aa_k,\mu_k)$ is $\sigma$-finite, \eqref{18.05.2011.01} holds and $\Delta_k$ is an isometry, then we call the target space $(X,\aa,\mu)$ {\em isometric} and write $(X,\aa,\mu)=\ILIM\psnmu$.
\begin{lem}\label{lpisom}
Let $(X,\aa,\mu)=\LIM\psnmu$ and $k\in\N$. If $(X_k,\aa_k,\mu_k)$ is $\sigma$-finite, condition \eqref{18.05.2011.01} is satisfied and $\Delta_k$ is an isometry on $L^2(\mu_k)$, then it is an isometry on $L^p(\mu_k)$ for all $1\Le p<\infty$ and the measure $\mu$ is $\sigma$-finite.
\end{lem}
\begin{proof}
First, we observe that $\Delta_k$ is an isometry on $L^p(\mu_k)$ if and only if the Radon-Nikodym derivative $\tfrac{\D\mu\circ(\delta^k)^{-1}}{\D\mu_k}=1$ almost everywhere (with respect to $\mu_k$). This follows directly from
the equality
    \begin{align*}
    \int_{X_k}\tfrac{\D\mu\circ(\delta^k)^{-1}}{\D\mu_k}\,f \D\mu_k=\int_{X}f\circ \delta^k\D\mu,
    \end{align*}
valid for every  $\aa_k$-measurable non-negative function $f$ (it holds by the measure transport theorem). Hence $\Delta_k$ is an isometry on $L^p(\mu_k)$ for every $1\Le p<\infty$. Now, by $\sigma$-finiteness of $\mu_k$, there exists $\{\sigma_n\}_{n=1}^\infty\subseteq (\aa_k)_{\mu_k}$ such that $\sigma_n\nearrow X_k$ as $n\to\infty$. Clearly, if $\Delta_k$ is isometric, then
\begin{align*}
\mu\big((\delta_k)^{-1}(\sigma_n)\big)=\int_X \chi_{\sigma_n}\circ \delta_k\,\D\mu=\int_{X_k} \chi_{\sigma_n}\,\D\mu_k=\mu_k(\sigma_n),\quad n\in\N.
\end{align*}
Hence  $\{(\delta_k)^{-1}(\sigma_n)\}_{n=1}^\infty\subseteq (\aa)_{\mu}$. Since $(\delta_k)^{-1}(\sigma_n)\nearrow X$ as $n\to\infty$, we see that $\mu$ is $\sigma$-finite.
\end{proof}
\begin{rem}
It is worth noticing, that if $\{(X_n,\aa_n,\mu_n)\}_{n\in\N}$ and $(X,\aa,\mu)$ satisfy all the conditions of definition of a target space except
condition $(\mathtt{P}_7)$ and all the operators $\varDelta_k$ are isometries, then $(\mathtt{P}_7)$ is automatically satisfied. Indeed, take $m\Ge n$.
Then, by $(\mathtt{P}_5)$ and Lemma \ref{lpisom}, for every $\sigma\in(\aa_n)_{\mu_n}$ we have
    \begin{align*}
    \mu\big((\delta^n)^{-1}(\sigma)\big)
    &=\mu\big((\delta_m^n\circ\delta^m)^{-1}(\sigma)\big)=\mu\Big((\delta^m)^{-1}\big((\delta_m^n)^{-1}(\sigma)\big)\Big)\notag\\
    &=\int_X|\Delta_m\chi_{(\delta_m^n)^{-1}(\sigma)}|\D\mu=\int_{X_m}|\chi_{(\delta_m^n)^{-1}(\sigma)}|\D\mu_m\\
    &=\mu_m\big((\delta_m^n)^{-1}(\sigma)\big)\notag.
    \end{align*}
This and $\sigma$-finiteness of $\mu_m$ prove the claim.
\end{rem}
Let $(X,\aa,\mu)=\LIM\psnmu$. Assume that $(X_k,\aa_k,\mu_k)$ is $\sigma$-finite for every $k\in\N$. Set $\frcn=\Delta_n L^2(\mu_n)$, $n\in\N$, and $\frc=\bigcup_{n\in\N}\frcn$. Members of $\frc$  are called {\em cylinder functions}. By $\frx$ we denote the of set all characteristic functions of sets from $(\aa)_\mu$, regarded as a linear subspace of $L^2(\mu)$, while $\frx_c$ stands for the intersection of $\frx$ and $\frc$.  Throughout what follows
$\frs$ denotes the linear span of $\frx_c$.
\begin{rem}
Suppose that $(X,\aa,\mu)=\ILIM\psnmu$, $(X,\aa,\mu)$. Clearly, characteristic functions of sets of finite measure are linearly dense in $L^2(\mu)$. By \cite[Approximation Theorem 1.3.11]{ash-dol}, this and condition $(\mathtt{P}_6)$ imply that $\frs$ is linearly dense in $L^2(\mu)$ as well. In particular, cylinder functions are dense in $L^2(\mu)$.
\end{rem}
Now, let $(X,\aa,\mu)=\LIM\psnmu$. Let $k,l\in\N$ be such that $k\Le l$. Suppose that $(X_k,\aa_k,\mu_k)$ is $\sigma$-finite,
\begin{align}\label{delta1}
\mu_l\circ(\delta_l^k)^{-1}\ll \mu_k \ \text{ and }\ \frac{\D\mu_l\circ(\delta_l^k)^{-1}}{\D\mu_k}\in L^{\infty}(\mu_k).
\end{align}
Analogously to operators $\Delta_n$, we may define bounded operators
\begin{align*}
\Delta_k^l:L^2(\mu_k)\ni f\mapsto f\circ\delta_l^k\in L^2(\mu_l).
\end{align*}
The projective system setting fits well together with inductive limits of $L^2$-spaces (see Lemma \ref{indlim} below). We will use this fact when implementing approximation procedure for a study of composition operators acting in $L^2$-spaces over measure spaces being isometric target spaces of projective systems.
\begin{lem}\label{indlim}
Let $(X,\aa,\mu)=\ILIM\psnmu.$ Then condition \eqref{delta1} is satisfied, operators $\Delta_k^l$ are isometries for all natural numbers $k\Le l$
and $L^2(\mu)$ is the inductive limit of $L^2(\mu_n)$.
\end{lem}
\begin{proof}
Set $\Lambda_k^l=\Delta_k^l$ and $\Lambda_k=\Delta_k$ for $k,l\in\N$ such that $k\Le l$. It is evident then that conditions
($\mathtt{I}_1$)-($\mathtt{I}_4$) are satisfied.
\end{proof}
{\bf Caution.} From now on we tacitly assume that if $(X,\aa,\mu)=\ILIM\psnmu$, then $L^2(\mu)=\ILIM L^2(\mu_n)$ with respect to the maps $\Lambda_k^l$ and $\Lambda_k$ as in the proof of Lemma \ref{indlim}.
\section{Composition operators and inductive limits over projective systems}
\subsection{Dense definiteness and boundedness of $\LIM C_{A_n}$}
In this part of the paper we are aiming to supply some quite natural assumptions which would imply that the inductive limit operator $\LIM C_{A_n}$ of composition operators is densely defined or bounded. We begin by describing the domain of $\LIM C_{A_n}$.
\begin{lem}\label{domlim}
Let $(X,\aa,\mu)=\ILIM\psnmu$. Suppose that for every $l\in\N$, $A_l$ is a nonsingular $\aa_l$-measurable transformation of $X_l$ such that $\hsf^{A_l}<\infty$ a.e. $[\mu_l]$. Then $\dom(\LIM C_{A_n})$ consists of all $\Lambda_k f\in \dom_\infty$ such that for every $\varepsilon>0$ the following condition
\begin{align}\label{condit}
\|\esf_m\big(\Lambda_n^m ((\Lambda_k^n f)\circ A_n)\big)\circ A_m^{-1} - \Lambda_k^m f\|_{L^2(\hsf^{A_m}\D\mu_m)}\Le \varepsilon, \quad m\Ge n\Ge M.
\end{align}
is satisfied with sufficiently large $M\in \N$.
\end{lem}
\begin{proof}
Since $(X,\aa,\mu)$ is an isometric inductive limit of $\{(X_n,\aa_n,\mu_n)\}_{n\in\N}$, by \eqref{RN} and \eqref{fifi}, we have
    \begin{align*}
    &\|\Lambda_m C_{A_m} \Lambda_k^m f-\Lambda_n C_{A_n}\Lambda_k^n f\|_{L^2(\mu)}^2
    =\int_{X_m}|\Lambda_n^m \big((\Lambda_k^n f)\circ A_n\big) -(\Lambda_k^m f)\circ A_m|^2\D\mu_m\\
    &=\int_{X_m}|\esf_m \big( \Lambda_n^m ((\Lambda_k^n f)\circ A_n)\big)\circ A_m^{-1} -\Lambda_k^m f|^2\hsf^{A_m}\D\mu_m, \quad m\Ge n\Ge k,
    \end{align*}
for every $f\in\dom(C_{A_k})$. Hence, the claim follows from the definition of $\LIM C_{A_n}$.
\end{proof}
\begin{rem}
Regarding Lemma \ref{domlim}, it is worth pointing out that it may happen that $\LIM C_{A_n}$ is densely defined in $L^2(\mu)$ but none of the $C_{A_l}$, $l\in\N$, is densely defined. For example, if $\hh=\hh_n=\ell^2(\N)$ and $\phi_n(x)=\min\{n,x\}$ for $n,x\in\N$, then $\dom(C_{\phi_n})=\lin\{e_1,\ldots,e_{n-1}\}$, where $\{e_n\}_{n=1}^\infty$ is a standard orthonormal basis of $\ell^2(\N)$. On the other hand, $\dom(\LIM C_{\phi_n})=\lin\{e_n\colon n\in \N\}$ which is dense in $\hh$. Consequently, in this particular situation the assumption $\hsf^{A_l}<\infty$ a.e. $[\mu_l]$ for all $l\in\N$ is not satisfied (see by \eqref{dense}).
\end{rem}
Characteristic functions of cylinder sets with finite measure are the most elementary functions which we expect to belong to the domain of $\LIM C_{A_n}$. The conditions (i) and (ii) of Proposition \ref{cauchy} below turns out to be essential for this to happen.
\begin{prop}\label{cauchy}
Let $(X,\aa,\mu)=\ILIM\psnmu$. Suppose that $A_n$, $n\in\N$, is a nonsingular $\aa_n$-measurable transformation of $X_n$ such that $\hsf^{A_n}<\infty$ a.e. $[\mu_n]$. Let $k\in\N$ and $\sigma\in(\aa_k)_{\mu_k}$. Then $\Delta_k \chi_{\sigma}\in\dom(\LIM C_{A_n})$ if and only if the following two conditions are satisfied
\begin{enumerate}
\item[(i)] there exists $M\Ge k$ such that $\chi_{(\delta_n^k)^{-1}(\sigma)} \hsf^{A_n}\in L^1(\mu_n)$ for all $n\Ge M$
\item[(ii)] for every $\varepsilon>0$ there exists $N\Ge k$ such that
    \begin{align*}
    \mu\Big(\big(\delta_n^k\circ A_n\circ \delta^n\big)^{-1}(\sigma) \triangle \big(\delta_m^k\circ A_m\circ \delta^m\big)^{-1}(\sigma)\Big)\Le\varepsilon, \quad n,m\Ge N.
    \end{align*}
\end{enumerate}
\end{prop}
\begin{proof}
Set $f=\Delta_k \chi_{\sigma}$. Clearly, $f\in L^2(\mu)$. Since we have
    \begin{align*}
    \int \big|\Delta_k^n\chi_{\sigma} \circ A_n\big|^2\D\mu_n
        &=
    \mu_n\Big(A_n^{-1}\big((\delta_n^k)^{-1}(\sigma)\big)\Big)\\
        &=
    \int \chi_{(\delta_n^k)^{-1}(\sigma)}\hsf^{A_n} \D\mu_n, \quad n\Ge k,
    \end{align*}
we see that (i) is equivalent to $\Delta_k^n\chi_{\sigma_k} \in\dom(C_{A_n})$ for any $n\Ge M$. On the other hand, by \eqref{RN} and \eqref{fifi}, we have
\begin{multline*}
    \|\esf_m\big(\Lambda_n^m ((\Lambda_k^n \chi_\sigma)\circ A_n)\big)\circ A_m^{-1} - \Lambda_k^m \chi_\sigma\|_{L^2(\hsf^{A_m}\D\mu_m)}\\
    =
    \mu\Big(\big(\delta_n^k\circ A_n\circ \delta^n\big)^{-1}(\sigma) \triangle \big(\delta_m^k\circ A_m\circ \delta^m\big)^{-1}(\sigma)\Big),
    \quad m\Ge n \Ge k.
\end{multline*}
Therefore, condition (ii) is equivalent to condition \eqref{condit}. These two facts, in view of Lemma \ref{domlim}, imply the claim.
\end{proof}
\begin{cor}\label{denslim}
Let $(X,\aa,\mu)=\ILIM\psnmu$. Suppose that $A_n$, $n\in\N$, is a nonsingular $\aa_n$-measurable transformation of $X_n$.  If conditions {\rm (i)} and {\rm (ii)} of Proposition \ref{cauchy} are satisfied for all $\sigma\in(\aa_k)_{\mu_k}$ and $k\in\N$, then $\frs\subseteq \dom(\LIM C_{A_n})$.
\end{cor}
Clearly, one immediate consequence of the above is that, under assumptions of Corollary \ref{denslim}, the inductive limit $\LIM C_{A_n}$ is densely defined. The same happens if all the operators $C_{A_n}$, $n\in \N$,  are bounded and condition (ii) of Proposition \ref{cauchy} holds. Using similar arguments as in the proof of Proposition \ref{cauchy} we can obtain the following version of Corollary \ref{denslim}, which again yields dense definiteness of $\LIM C_{A_n}$.
\begin{cor}\label{denslim+}
Let $(X,\aa,\mu)=\ILIM\psnmu$ and let $A_n$, $n\in\N$, be a nonsingular $\aa_n$-measurable transformation of $X_n$. Assume there exists a sequence $\{Z_n\}_{n\in\N}$ of sets such that $Z_n\in(\aa_n)_{\mu_n}$, $(\delta^n)^{-1}(Z_n)\nearrow X$ as $n\to\infty$ and
\begin{align*}
\big\{\chi_{(\delta^k)^{-1}(Z_k)}\Delta_n\hsf^{A_n}\colon k, n\in\N\big\}\subseteq L^1(\mu).
\end{align*}
Suppose that for every $\sigma\in(\aa_k)_{\mu_k}$, $k\in\N$, and every $\varepsilon>0$ there exists $N\Ge k$ such that
    \begin{align}\label{densetilde}
    \mu\Big(\big(\delta_n^k\circ A_n\circ \delta^n\big)^{-1}(\sigma\cap Z_k) \triangle \big(\delta_m^k\circ A_m\circ \delta^m\big)^{-1}(\sigma\cap Z_k)\Big)\Le\varepsilon, \quad n,m\Ge N.
    \end{align}
Then $\tilde\frs\subseteq \dom(\LIM C_{A_n})$, where $\tilde\frs=\{\chi_{(\delta^k)^{-1}(Z_k)} f\colon f\in \frs, k\in\N\}$.
\end{cor}

The questions of boundedness of $\LIM C_{A_n}$ can be answered in the following way.
\begin{prop}\label{cauchy'}
Let $(X,\aa,\mu)=\ILIM\psnmu$. Suppose that $A_n$ is a nonsingular $\aa_n$-measurable transformation of $X_n$ for every $n\in\N$. If the following conditions are satisfied
\begin{enumerate}
\item[(i)] for all $k\in\N$, $\hsf^{A_k}\in L^\infty (\mu_k)$ $($or, equivalently, $C_{A_k}\in\bb(L^2(\mu_k))$$)$,
\item[(ii)] for every $k\in\N$ and $\sigma\in(\aa_k)_{\mu_k}$, the condition {\rm (ii)} of Proposition \ref{cauchy} holds,
\item[(iii)] for all $k\in\N$ there exist $C>0$ and $N\in\N$ such that
    \begin{align*}
    \mu_m\Big(\big(A_k\circ\delta_m^k\big)^{-1}(\sigma) \triangle \big(\delta_m^k\circ A_m\big)^{-1}(\sigma)\Big) \Le C\mu_k(\sigma),
    \quad \sigma\in(\aa_k)_{\mu_k}, m\Ge N,
    \end{align*}
\end{enumerate}
then $\frc\subseteq \dom(\LIM C_{A_n})$ and for every $k\in\N$, $(\LIM C_{A_n})|_{\frck}$ is bounded.
\end{prop}
\begin{proof}
Clearly, (i) and (ii) imply conditions (i) and (ii) of Proposition \ref{cauchy} an thus $\frs\subseteq\dom(\LIM C_{A_n})$. Now, we show that for a fixed $k\in\N$ we have
    \begin{align}\label{21.08.07.3}
    \sup_{n\in\N} \|C_{A_n}|_{\Lambda_k^n (L^2(\mu_k))}\|<\infty.
    \end{align}
By (iii) there exist $C>0$ and $N\in\N$ such that
    \begin{align*}
    \|\Lambda_mC_{A_m}\Lambda_k^m\chi_{\sigma}-\Lambda_kC_{A_k}\chi_{\sigma}\|^2
    &=\mu_m\big((A_k\circ\delta_m^k)^{-1}(\sigma)\triangle(\delta_m^k\circ A_m)^{-1}(\sigma)\big)\\ \notag
    &\Le C\mu_k(\sigma)
    =C\mu_N((\Lambda_N^k)^{-1}(\sigma))
    \end{align*}
holds for all $m\Ge N$ and $\sigma\in(\aa_k)_{\mu_k}$. This together with ($\mathtt{P}_5$), the fact that all $\Lambda_l$'s are isometries and \eqref{RN} implies\allowdisplaybreaks
    \begin{align*}
    \int\chi_{(\delta_n^N)^{-1}(\omega)} \hsf^{A_n}\D\mu_n
    &=\|C_{A_n}\Lambda_N^n\chi_{\omega}\|^2_{L^2(\mu_n)}\\
    &\Le\Big(\|C_{A_n}\Lambda_N^n\chi_{\omega}-\Lambda_k^n C_{A_k}\chi_\sigma\|_{L^2(\mu_n)}+\|\Lambda_k^n C_{A_k}\chi_\sigma\|_{L^2(\mu_n)}\Big)^2\\
    &\Le \Big(\sqrt{C}\,\|\chi_\omega\|_{L^2(\mu_N)}+\| C_{A_k}\|\|\chi_{\omega}\|_{L^2(\mu_N)}\Big)^2\notag\\
    &=\Big(\sqrt{C}+\|C_{A_k}\|\Big)^2\cdot \int \chi_\omega \D\mu_N\notag\\
    &=\Big(\sqrt{C}+\|C_{A_k}\|\Big)^2\cdot \int \chi_{(\delta_n^N)^{-1}(\omega)} \D\mu_n\notag.
    \end{align*}
for every $n\Ge N$ and $\omega=(\delta_N^k)^{-1}(\sigma)$, with $\sigma\in (\aa_k)_{\mu_k}$. Now, applying standard approximation argument and the fact that the family $\{\chi_{(\delta_n^k)^{-1}(\sigma)}\colon \sigma\in(\aa_k)_{\mu_k}\}$ is linearly dense in $\Lambda_k^n\big(L^1(\mu_k)\big)$ we obtain
    \begin{align}\label{21.08.07.2}
    \int |g|^2 \hsf^{A_n}\D\mu_n \Le \Big(\sqrt{C}+\|C_{A_k}\|\Big)^2\cdot \|g\|^2
    \end{align}
for every $n\Ge N$ and every $g\in \Lambda_k^n(L^2(\mu_k))$. This, if combined with \eqref{RN}, yields
    \begin{align*}
    \|C_{A_n}g\|^2 \Le \Big(\sqrt{C}+ \|C_{A_k}\|\Big)^2\|g\|^2,\quad n\Ge N,\ g\in \Lambda_k^n(L^2(\mu_k)).
    \end{align*}
This yields \eqref{21.08.07.3}. Employing Lemma \ref{indbound}\,(1) we conclude the proof.
\end{proof}
\begin{prop}\label{cauchy''}
Let $(X,\aa,\mu)=\ILIM\psnmu$. Suppose that $A_n$ is a nonsingular $\aa_n$-measurable transformation of $X_n$ for every $n\in\N$. If the following conditions are satisfied
\begin{enumerate}
\item[(i)] $\sup_{k\in \N}\|\hsf^{A_k}\|_{L^\infty(\mu_k)}<\infty$,
\item[(ii)] for every $k\in\N$ and $\sigma\in(\aa_k)_{\mu_k}$, the condition $($ii$)$ of Proposition \ref{cauchy} holds.
\end{enumerate}
then $\LIM C_{A_n}$ is closable and $\overline{\LIM C_{A_n}}\in \bb(L^2(\mu))$.
\end{prop}
\begin{proof}
By (i) and \eqref{norm} we see that for every $k\in\N$, $C_{A_k}\in\bb(L^2(\mu_k))$. Hence, by Proposition \ref{cauchy}, $\frs\subseteq\dom(\LIM C_{A_n})$ and thus $\LIM C_{A_n}$ is densely defined. Now, the claim follows from Lemma \ref{indbound}\,(2) and the fact that $\sup_{k\in\N}\|C_{A_n}\|<\infty$ (see \eqref{norm}).
\end{proof}
\subsection{Well-definiteness and boundedness of $C_A$}
Now we address the question of when the composition operator $C_A$ acting in an inductive limit of $L^2$-spaces is well-defined and bounded. We do it by relating $C_A$ to an inductive limit $\LIM C_{A_n}$.

We begin by recalling well-known criteria for $*$-weak compactness of a family $\ff$ contained in $L^p(\mu)$, $1\Le p<\infty$; the case $p=1$ follows directly from the compactness criterion of Dunford-Pettis theorem (cf. \cite[Chapter II, Theorem T23]{mey}) combined with the de la Vall\'e-Poussin's theorem (cf. \cite[Lemma 6.5.6]{ash-dol}) and the case $p>1$ is a consequence of the Banach-Alaoglu theorem.
\begin{lem}\label{dpvp}
Let $\ff\subseteq L^p(\mu)$ be countable and $1\Le p<\infty$. If one of the following conditions is satisfied$:$
\begin{enumerate}
\item[(i)] $p=1$, $\mu$ is a finite measure and $\frm(\ff)\neq\varnothing$,
\item[(ii)] $p>1$ and $\ff$ is uniformly bounded in $L^p(\mu)$,
\end{enumerate}
then $\ff$ is $*$-weakly compact.
\end{lem}
Now, we show that absolute continuity of the measures $\mu_n$ and $\nu_n$ transfers onto theirs target measures $\mu$ and $\nu$.
\begin{lem}\label{absgen1}
Let $(X,\aa,\mu)=\ILIM\psnmu$ and $(X,\aa,\nu)=\LIM\psnnu$. Suppose that $\nu_n\ll\mu_n$ for all $n\in\N$. If there exists a sequence $\{Y_n\}_{n\in\N}$ of cylinder sets such that $Y_n\nearrow X$ as $n\to\infty$, $\big\{\chi_{Y_k}\Delta_n(\D\nu_n/\D\mu_n)\colon n\in\N\big\}\subseteq L^1(\mu)$ for every $k\in\N$ and $\frm\Big(\big\{\chi_{Y_k}\Delta_n(\D\nu_n/\D\mu_n)\colon n\in\N\big\}\Big)\neq\varnothing$ for every $k\in \N$, then $\nu\ll\mu$.
\end{lem}
\begin{proof}
For $k\in\N$, let $\aa_{Y_k}$ stand for the $\sigma$-algebra $\{\omega\in\aa\colon \omega\subseteq {Y_k}\}$,
let $\mu_{Y_k}$ denote a restriction of $\mu$ to $\aa_{Y_k}$ and let $h_n|_{Y_k}$, $n\in\N$, be a restriction of $\Delta_n(\D\nu_n/\D\mu_n)$ to ${Y_k}$. By Lemma \ref{dpvp} (with $p=1$), the sequence $\{h_n\}_{n\in\N}$ has a subsequence $\{h_{n(k,1)}\}_{k\in\N}$ such that $\{h_{n(k,1)}|_{Y_1}\}_{k\in\N}$ converges $*$-weakly to a function $h^{Y_1}\in L^{1}({Y_1},\aa_{Y_1},\mu_{Y_1})$. The same argument implies that $\{h_{n(k,1)}\}_{k\in\N}$ has a subsequence $\{h_{n(k,2)}\}_{k\in\N}$ such that $\{h_{n(k,2)}|_{Y_2}\}_{k\in\N}$ converges $*$-weakly to a function $h^{Y_2}\in L^{1}({Y_2},\aa_{Y_2},\mu_{Y_2})$.
Clearly, we have
    \begin{align}\label{07.09.01}
    h^{Y_2}|_{Y_1}=h^{Y_1}\quad \text{a.e.\ $[\mu]$}.
    \end{align}
If we repeat the argument $l$ times we get a subsequence $\{h_{n(k,l)}\}_{k\in\N}$ of a sequence $\{h_{n(k,l-1)}\}_{k\in\N}$ such that $\{h_{n(k,l)}|_{Y_{l}}\}_{k\in\N}$ is converging $*$-weakly to a function $h^{Y_l}$ contained in $L^{1}({Y_l},\aa_{Y_l},\mu_{Y_l})$.
Moreover, by \eqref{07.09.01} we have
    \begin{align}\label{07.09.02}
    h^{Y_l}|_{Y_k}=h^{Y_k}\quad \text{a.e.\ $[\mu]$,}\quad  k\Le l.
    \end{align}
Since $\bigcup_{n\in\N}Y_n=X$, we can use \eqref{07.09.02} as to obtain a function $h\colon X\to \overline{\R}_+$
which is $\aa$-measurable and satisfies
    \begin{align}\label{07.09.03}
    h|_{Y_l}=h^{Y_l}\quad \text{a.e.\ $[\mu]$},\quad l\in\N.
    \end{align}
Now, for $m\in\N$, let $\sigma\in\aa_m$. Set $\tau=\big(\delta^m\big)^{-1}(\sigma)$. Fix $N\in \N$. There exists $l\in\N$ and $Z_N\in\aa_l$ such that $Y_N=(\delta^l)^{-1}(Z_N)$. Without loss of generality we may assume that $l\Ge m$. Define $\tau_N=\big(\delta^m\big)^{-1}(\sigma)\cap Y_N$ and $\omega_N=\big(\delta^m_{l}\big)^{-1}(\sigma)\cap Z_N$. Then, by \eqref{07.09.03}, ($\mathtt{P}_5$) and ($\mathtt{P}_7$), we gather that
\begin{align*}
    \int_{\tau_N} h\D\mu
        &=
    \lim_{k\rightarrow\infty}\int_{\tau_N}\Delta_{n(k,N)}h_{n(k,N)}\D\mu\\
        &=
    \lim_{k\rightarrow\infty}\int\Delta_{n(k,N)}\left(\chi_{(\delta^l_{n(k,N)})^{-1}(\omega_N)}h_{n(k,N)}\right)\D\mu\\
        &=
    \lim_{k\rightarrow\infty} \int_{(\delta^l_{n(k,N)})^{-1}(\omega_N)}h_{n(k,N)}\D\mu_{n(k,N)}\\
        &=
    \lim_{k\rightarrow\infty}\int_{(\delta^l_{n(k,N)})^{-1}(\omega_N)}\D\nu_{n(k,N)}\\
        &=
    \nu(\tau_N).
\end{align*}
This, ($\mathtt{P}_6$) and \cite[Theorem 10.3]{bil} imply that $\nu(\tau)=\int\chi_\tau h\D\mu$ for every $\tau\in\aa$ which proves our claim.
\end{proof}
\begin{rem}
Regarding Lemma \ref{absgen1}, we mention the paper \cite{kak} where necessary and sufficient conditions for equivalence (in sense of absolute continuity) of tensor product measures are supplied. Those conditions, however, cannot be applied in our context (because, in general, measures of the form $\mu\circ A^{-1}$ are not tensor products).
\end{rem}
Now, suppose $A$ is an $\aa$-measurable transformation of $X$. Consider the following condition:
\begin{align}\label{DAG}
   \begin{minipage}{32em}
    For all $k\in\N$ and $\sigma\in\aa_k$ there is
$$
    \lim_{n\to\infty}\mu\Big((\delta_n^k\circ A_n\circ \delta^n)^{-1}(\sigma)
        \triangle
    \big(\delta^k \circ A\big)^{-1}(\sigma)\Big)=0.
$$
   \end{minipage}
\end{align}
Next theorem shows that if $A$ can be approximated (in a sense of conditions \eqref{DAG}) by a sequence of $\aa_n$-measurable transformations $A_n$, then the composition operator $C_A$ is exactly the inductive limit of operators $C_{A_n}$.
\begin{thm}\label{absconA}
Let $(X,\aa,\mu)=\ILIM\psnmu$. Let $A_n$ be an $\aa_n$-measurable transformation of $X_n$ for $n\in\N$ and let $A$ be an $\aa$-measurable transformation $X$ such that condition \eqref{DAG} is satisfied.
\begin{enumerate}
\item[(i)] Then $(X,\aa,\mu\circ A^{-1})=\LIM(X_n,\aa_n,\mu_n\circ A_n^{-1})$.
\item[(ii)] If $\mu_n\circ A^{-1}_n\ll\mu_n$ for all $n\in\N$ and there exists a sequence $\{Y_n\}_{n\in\N}$ of cylinder sets such that $Y_n\nearrow X$ as $n\to\infty$, $\big\{\chi_{Y_k}\Delta_n\hsf^{A_n}\colon n\in\N\big\}\subseteq L^1(\mu)$ for every $k\in\N$ and $\frm\Big(\big\{\chi_{Y_k}\Delta_n\hsf^{A_n}\colon n\in\N\big\}\Big)\neq\varnothing$ for every $k\in \N$, then $C_A$ is densely defined and closed in $L^2(\mu)$ and $C_A=\overline{\LIM C_{A_n}|_{\tilde\frs}}$, where $\tilde\frs=\{\chi_{Y_k} f\colon f\in \frs, k\in\N\}$.
\end{enumerate}
\end{thm}
\begin{proof}
(i) Let $k,n\in\N$ and let $\sigma\in\aa_k$ be such that either $\mu\circ A^{-1} \big((\delta^k)^{-1}(\sigma)\big)<\infty$ or $\mu_n\circ A_n^{-1} \big((\delta_n^k)^{-1}(\sigma)\big)<\infty$. Then we have
    \begin{multline*}
    |\mu\circ A^{-1}\big((\delta^k)^{-1}(\sigma)\big)-\mu_n\circ A_n^{-1}\big((\delta_n^k)^{-1}(\sigma)\big)|=\\
    =\mu\Big((\delta_n^k\circ A_n\circ \delta^n)^{-1}(\sigma) \triangle \big(\delta^k \circ A\big)^{-1}(\sigma)\Big)=0.
    \end{multline*}
This and condition \eqref{DAG} imply that $(X,\aa,\mu\circ A^{-1})=\LIM(X_n,\aa,\mu_n\circ A_n^{-1})$.

(ii) Using \eqref{DAG}, the well-known inequality $\mu(\sigma_1\triangle\sigma_2) \Le \mu(\sigma_1\triangle\sigma_3)+\mu(\sigma_3\triangle\sigma_2)$, and Corollary \ref{denslim+} we deduce that $\tilde\frs\subseteq\dom(\LIM C_{A_n})$.

Lemma \ref{absgen1} and (i) imply that $\mu\circ A^{-1}\ll\mu$, which implies that $C_A$ is well-defined and closed operator in $L^2(\mu)$. Now we prove that $\tilde\frs\subseteq\dom(C_A)$. Fix $N\in\N$. There exists $l\in \N$ and $Z_N\in\aa_l$ such that $Y_N=(\delta^l)^{-1}(Z_N)$. By (i), for all $k\in\N$ and $\sigma\in\aa_k$  we have
    \begin{align}\label{31.08.01}
    \mu \circ A^{-1}\Big(\big(\delta^k\big)^{-1}(\sigma)\cap Y_N\Big)
        =
    \lim_{m\rightarrow\infty}\mu_m\circ A_m^{-1}\Big(\big(\delta_m^k\big)^{-1}(\sigma)\cap (\delta_m^l)^{-1}(Z_N)\Big).
    \end{align}
Fix $k\in \N$ and $\sigma\in(\aa_k)_{\mu_k}$. Let $m\in\N$ satisfy $k,l\Le m$. Set $\varOmega=\big(\delta^k\big)^{-1}(\sigma)\cap Y_N$.
Then we have
    \begin{align}\label{07.09.04}
    \mu_m\circ A_m^{-1}\Big(\big(\delta_m^k\big)^{-1}(\sigma)\cap (\delta_m^l)^{-1}(Z_N)\Big)=\int_{\varOmega} \Delta_m \hsf^{A_m}\D\mu.
    \end{align}
Let $G\in \frm\Big(\big\{\chi_{Y_N}\Delta_n\hsf^{A_n}\colon n\in\N\big\}\Big)$. There exists a non-negative real number $t_0$ such that $\big\{t\in[0,\infty)\colon G(t)\Ge t\big\}=[t_0,\infty)$. Also, for every $m\in \N$, there exist sets $\varTheta_1^m, \varTheta_2^m\in \aa$ such that $X=\varTheta_1^m\cup \varTheta_2^m$, $\varTheta_1^m\cap \varTheta_2^m=\varnothing$ and $\Delta_m \hsf^{A_m}(x)\Le t_0$ for $\mu$-a.e. $x\in \varTheta_1^m$ and $\Delta_m \hsf^{A_m}(x)\Ge t_0$ for $\mu$-a.e. $x\in \varTheta_2^m$. Hence we have
    \begin{align*}
    \int_{\varOmega} \Delta_m \hsf^{A_m}\D\mu
        &=
    \int_{\varOmega\cap X_1^m} \Delta_m \hsf^{A_m}\D\mu+\int_{\varOmega\cap X_2^m} \Delta_m \hsf^{A_m}\D\mu\\
        &\Le
    t_0\mu(\varOmega)+ \sup_{m\in \N} \int_{Y_N} G\big(\Delta_m \hsf^{A_m}\big)\D\mu.
    \end{align*}
This, \eqref{31.08.01} and \eqref{07.09.04} imply that $\chi_{\varOmega} \circ A\in L^2(\mu)$, which means that $\chi_{\varOmega}\in \dom(C_A)$. Since $k\in N$, $\sigma\in(\aa_k)_{\mu_k}$ and $N\in \N$ can be arbitrarily chosen, we see that $\tilde\frs\subseteq\dom(C_A)$.  This and $\sigma$-finiteness of $\mu$ imply that $C_A$ is densely defined.

Clearly, by \eqref{31.08.01} we have
    \begin{align}\label{05.09.01}
    \LIM C_{A_n} (\chi_{Y_N} f)=C_A (\chi_{Y_N} f),\quad f\in \frs,\, N\in \N.
    \end{align}
To conclude the proof it is sufficient to show that $\{\chi_{Y_N} f\colon f\in \frs, N\in\N\}$ is a core for $C_A$. For this, take $f\in\dom(C_A)$. It cause no loss of generality to assume that $f$ is non-negative. Since $\frs$ is dense in $L^2(\mu)$ and $Y_n\nearrow X$ as $n\to\infty$, there exists a sequence $\{f_n\}_{n\in\N}\subseteq \frs$ and a monotonically increasing mapping $\alpha\colon \N\to\N$ such that for $\mu$-a.e. $x\in X$, $(\chi_{Y_{\alpha(n)}} f_n)(x)\nearrow f(x)$ as $n\to\infty$. This implies that for $\mu$-a.e. $x\in X$, $(\chi_{Y_{\alpha(n)}}f_n)(A (x))\nearrow f(A(x))$ as $n\to\infty$. By the Lebesgue's monotone convergence theorem, $\{\chi_{Y_{\alpha(n)}} f_n\}_{n\in\N}$ and $\{(\chi_{Y_{\alpha(n)}} f_n)\circ A\}_{n\in\N}$ converge to $f$ and $f\circ A$ in $L^2(\mu)$, respectively. Therefore, we obtain $C_A\subseteq \overline{C_A|_{\tilde\frs}}$. This together with the fact that $C_A$ is closed implies that $C_A= \overline{C_A|_{\tilde\frs}}$. Using \eqref{05.09.01} we complete the proof.
\end{proof}
Combining Theorem \ref{absconA} and Lemma \ref{indbound} we obtain a criterion for boundedness of $C_A$ written in terms of composition operators $C_{A_n}$, $n\in\N$.
\begin{thm}\label{sufgen}
Let $(X,\aa,\mu)=\ILIM \psnmu$. Let $A_n$ be an $\aa_n$-measurable transformation of $X_n$ such that $\mu_n\circ A^{-1}_n\ll\mu_n$ for every $n\in\N$. Let $A$ be an $\aa$-measurable transformation $X$. If condition \eqref{DAG} holds and
\begin{align}\label{sup}
\sup_{n\in\N}\|\hsf^{A_n}\|_{L^\infty(\mu_n)}<\infty
\end{align}
then $C_A\in\bb\big(L^2(\mu)\big)$ and $C_A=\overline{\LIM C_{A_n}}.$
\end{thm}
\begin{proof}
By $\sigma$-finiteness of $\mu$ there exists a sequence $\{Y_k\}_{k\in\N}\subseteq (\aa)_{\mu}$ of cylinder sets such that $Y_k\nearrow X$ as $n\to\infty$. Clearly, \eqref{sup} implies that for every $k\in\N$ we have $\big\{\chi_{Y_k}\Delta_n\hsf^{A_n}\colon n\in\N\big\}\subseteq L^1(\mu)$ and $\frm\Big(\big\{\chi_{Y_k}\Delta_n\hsf^{A_n}\colon n\in\N\big\}\Big)\neq\varnothing$. By Theorem \ref{absconA} the operator $C_A$ is densely defined operator in $L^2(\mu)$ and $C_A=\overline{\LIM C_{A_n}|_{\tilde\frs}}$, where $\tilde\frs=\{\chi_{Y_k} f\colon f\in \frs, k\in\N\}$. Since condition \eqref{DAG} yields condition (ii) of Proposition \ref{cauchy} with any $\sigma\in\aa_k$ and $k\in\N$ (see proof of Theorem \ref{absconA}(i)), the operator $\overline{\LIM C_{A_n}}$ is bounded due to Proposition \ref{cauchy''}. This concludes the proof.
\end{proof}
\section{Examples and Applications}
In this part of the paper we demonstrate how inductive techniques of Theorems \ref{absconA} and \ref{sufgen} can be used when investigating composition operators in more concrete situations. We include some illustrative examples.

First, we provide a version of Theorem \ref{sufgen} in the context of $L^2$-space with respect to the gaussian measure on $\R^\infty$. Recall that the gaussian measure $\munfty$ on $\R^\infty$ is the tensor product measure $\munfty=\mathtt{g}\D\M_1\otimes \mathtt{g}\D\M_1\otimes\ldots$, where $\mathtt{g}(x)=\frac{1}{\sqrt{2\pi}} \exp({-\frac{x^2}{2}})$ for $x\in\R$. By the $n$-dimensional gaussian measure, $n\in\N$, we understand the measure $\MUN{n}$ given by $\D\MUN{n}=\frac{1}{(\sqrt{2\pi})^n}\, \exp({-\frac{x_1^2+\ldots +x_n^2}{2}}) \D\M_n$. Clearly, if $\delta^n$ and $\delta^n_k$ are the projections from $\R^\infty$ and $\R^k$ (respectively) onto $\R^n$, we have
    \begin{align*}
    \big(\R^\infty,\bor{\R^\infty},\munfty\big)
        =
    \ILIM\Big(\R^{n},\bor{\R^{n}},\MUN{n}\Big)
    \end{align*}
and consequently
    \begin{align*}
    L^2(\munfty)
        =
    \ILIM L^2(\MUN{n}).
    \end{align*}
Let $(a_{ij})_{i,j\in\N}$ be a matrix with real entries. We say that a transformation $A$ of $\R^\infty$ is induced by $(a_{ij})_{i,j\in\N}$ if the following condition holds\footnote{We assume that all the series $\sum_{j\in\N} a_{kj}\,x_j$, $k\in\N$, are convergent.}
\begin{align*}
\A(x_1,x_2,\ldots)=\big(\sum_{j\in\N} a_{1j}\,x_j,\sum_{j\in\N} a_{2j}\,x_j,\ldots\big),\quad (x_1, x_n, \ldots)\in\R^\infty.
\end{align*}
In an analogical way we define a transformation $A$ of $\R^n$ to be induced by a finite dimensional matrix $(a_{ij})_{i,j=1}^n$.

In view of \eqref{norminv}, Theorem \ref{absconA} can be rewritten in the present context in the following manner. Below, $\|\cdot\|$ denotes\footnote{For simplicity we do not make the dependence of $\|\cdot\|$ on $n$ explicit.} the Euclidean norm on $\R^n$.
\begin{cor}\label{gaussdense+}
Let $\A$ be a transformation of $\R^\infty$ induced by a matrix $(a_{ij})_{i,j\in\N}$. Let $\A_n$, $n\in\N$, be the linear transformation of $\R^n$ induced by the matrix $(a_{ij})_{i,j=1}^n$. If the following conditions are satisfied:
\begin{enumerate}
\item[(i)] for every $n\in\N$, $A_n$ is invertible,
\item[(ii)] for every $j\in\N$ there is $K\in\N$ such that $a_{j,k}=0$ for all $k\Ge K$,
\item[(iii)] there exists a sequence $\{\sigma_k\}_{k\in\N}$ of sets $\sigma_k\in\bor{\R^k}$ such that
\begin{itemize}
\item $\sigma_k\times \R^\infty\nearrow \R^\infty$ as $k\to\infty$,
\item $\chi_{\sigma_k\times \R^{n-k}}\cdot \exp \tfrac12 \big( \|\cdot\|^2-\|A_{n}^{-1}(\cdot)\|^2\big)\in L^1(\MUN{n})$ for every $n\Ge k$,
\item $\frm\bigg( \Big\{\big|\det A_{n}^{-1}\big|\cdot \chi_{\sigma_k\times \R^{n-k}}\cdot \exp \tfrac12 \big( \|\cdot\|^2-\|A_{n}^{-1}(\cdot)\|^2\big)\colon n\Ge k\Big\}\bigg)\neq\varnothing$,
\end{itemize}
\end{enumerate}
then $C_A$ is densely defined operator in $L^2(\munfty)$ and $C_A=\overline{\LIM C_{A_n}|_{\tilde\frs}}$, where $\tilde\frs=\{\chi_{\sigma_k\times \R^\infty} f\colon f\in \frs, k\in\N\}$.
\end{cor}
An example of densely defined composition operator in $L^2(\munfty)$ is presented below.
\begin{exa}\label{exaunbounded2}
Let $A\colon\R^\infty\to\R^\infty$ be induced by the matrix $(a_{ij})_{i,j\in\N}$ given by
\begin{align*}
a_{ij}=\left\{%
\begin{array}{ll}
    1 , & \hbox{for $i=j$;} \\
    2^{-(i+j)}, & \hbox{$j=i+1$};\\
    0       , & \hbox{otherwise.} \\
\end{array}%
\right.
\end{align*}
Let $A_n\colon\R^n\to\R^n$ be induced $(a_{ij})_{i,j=1}^n$, $n\in\N$. Clearly, the conditions (i) and (ii) of Corollary \ref{gaussdense+} are satisfied. It is elementary to show that $\sup_{n\in\N}\|A_n\|_{\bb(\R^n)}<\sqrt{2}$. This implies that there exists a positive real $t_0$ such that $\|A_n\|_{\bb(\R^n)}^2\Le 2-2t_0$ for every $n\in\N$. Thus we have
    \begin{align*}
    \int_{\R^n}\Big(\E^{\nicefrac{(\|x\|^2-\|A_{n}^{-1}x\|^2)}{2}}\Big)^2\MUN{n}(\D x)
        &=
    (2\pi)^{-\nicefrac{n}{2}}\int_{\R^n}\E^{\frac12\|x\|^2-\|A_n^{-1}x\|^2}\M_n(\D x)\\
        &\Le
    (2\pi)^{-\nicefrac{n}{2}}\int_{\R^n}\E^{\|A_n^{-1}x\|^2(\frac12\|A_n\|^2-1)}\M_n(\D x)\\
        &\Le
    (2\pi)^{-\nicefrac{n}{2}}\int_{\R^n}\E^{-t_0\|A_n^{-1}x\|^2}\M_n(\D x)\\
        &\Le
    (2\pi)^{-\nicefrac{n}{2}}\int_{\R^n}\E^{-\frac{t_0}{2}\|x\|^2}\M_n(\D x),\quad n\in\N.
    \end{align*}
Since $\det A_n=1$ for every $n\in\N$, we see that the condition (iii) of Corollary \ref{gaussdense+} is satisfied (with $G(t)=t^2$ and $\sigma_k=\R^k$). Hence, by Corollary \ref{gaussdense+}, $C_A$ is densely defined in $L^2(\munfty)$ and $C_A=\overline{\LIM C_{A_n}|_{\frs}}$. It is worth noticing that for every $n\in\N$ the norm of $A_n$ in $\bb(\R^n)$ is greater than 1, which (in view of \cite[Proposition 2.2]{sto}) implies that $C_{\A_n}$ is not bounded on $L^2(\MUN{n})$.
\end{exa}
Now we supply a tractable criterion for boundedness of a composition operator $C_A$ in $L^2$-space over an infinite tensor product of arbitrary probability spaces. For this we consider $\{(\Omega_n,\varSigma_n,P_n)\}_{n\in\N}$, a sequence of probabilistic spaces. Let $X_m= \Omega_1\times\ldots\times \Omega_m$, $m\in\N$, and $X=\Omega_1\times \Omega_2\times \ldots$. For $m\Le n$, let $\delta^n$ and $\delta^m_n$ denote the projection from $X$ and $X_n$, respectively, onto $X_m$. Let $\aa_m=\sigma(\varSigma_1\times\ldots\times\varSigma_m)$, $m\in\N$, and $\aa=\sigma\big(\{(\delta^m)^{-1}(\sigma)\colon \sigma\in\aa_m,\, m\in\N\}\big)$. Finally, let $\mu_m= P_1\otimes\ldots\otimes P_m$, $m\in\N$, and $\mu= P_1\otimes P_2\otimes\dots$ (cf. \cite[Section III-3]{nev}). Clearly, we have $(X,\aa,\mu)=\ILIM (X_n,\aa_n,\mu_n)$. It is well-known that
    \begin{align}\label{unitequiv}
    L^2(\mu)\simeq L^2(\mu_n)\otimes L^2(\otimes_{k=n+1}^\infty P_k),\quad n\in\N.
    \end{align}
(In the display above, $"\simeq "$ denotes unitary equivalence.)
Under all those circumstances, by Theorem \ref{sufgen}, we get the following criterion.
\begin{prop}\label{sufbten}
Let  $(X,\aa,\mu)$ and $\{(X_n,\aa_n,\mu_n)\}_{n\in\N}$ be as above. Let $A$ be an $\aa$-measurable transformation of $X$ and $\A_m$, $m\in\N$, be an $\aa_m$-measurable transformation of $X_m$. If the following conditions are satisfied
\begin{enumerate}
\item[(i)] for every $k\in\N$ there is $M\Ge k$ such that $\delta^k\circ A=\delta^k_n\circ A_n \circ \delta^n$ a.e.\ $[\mu]$ for all $n\Ge M$,
\item[(ii)] for every $n\in\N$, $\mu_n\circ\A_n^{-1}\ll\mu_n$,
\item[(iii)] $\sup_{n\in\N}\|\hsf^{\A_n}\|_{L^\infty(\mu_n)}<\infty$,
\end{enumerate}
then $C_\A\in \bb\big(L^2(\mu)\big)$. Moreover, we have
\begin{enumerate}
\item[(iv)] $\|C_\A\|^2=\sup_{n\in\N}\|\hsf^{\A_n}\|_{L^\infty(\mu_n)},$
\item[(v)] the operator $C_\A$ is the limit in the strong operator topology of the sequence $\{C_{\A_n}\otimes I_n\}_{n\in\N}$,  where $I_n$ is the identity operator on $L^2(\otimes_{k=n+1}^\infty P_k)$.
\end{enumerate}
\end{prop}
\begin{proof}
We infer from (ii), (iii) and \eqref{norm} that for every $n\in\N$, $C_{\A_n}\in \bb\big(L^2(\mu_n)\big)$ and
    \begin{align}\label{20.05.2011.01}
    \sup_{k\in\N}\|C_{\A_k}\|= M<\infty.
    \end{align}
By \eqref{unitequiv}, (i) and \eqref{20.05.2011.01}, the sequence $\{C_n\}_{n\in\N}$ of operators given by
\begin{align*}
C_n=C_{\A_n}\otimes I_n,
\end{align*}
is convergent in the strong operator topology to a bounded operator $C$ on $L^2({{\mu}})$. Since (i) yields condition \eqref{DAG}, the operator $C_\A$ is well-defined by Theorem \ref{sufgen}. Clearly, by (i), $C_\A$ and $C$ coincide on cylinder functions, which implies that $C_\A=C$. This completes the proof.
\end{proof}
A nontrivial example of a bounded composition operator acting in an $L^2$-space over the product of probabilistic measures (different than $L^2(\munfty)$) is presented below.
\begin{exa}\label{exatriangle1}
Let $\{\alpha_i\}_{i\in\N}\subseteq(1,\infty)$, $\{\rho_i\colon\R\to(0,\infty)\ |\ i\in\N\}$, and $\{M_{i}\}_{i\in\N}\subseteq(0,\infty)$ satisfy
\begin{itemize}
\item[(i)] $\prod_{i=1}^\infty\alpha_i<\infty$,
\item[(ii)] for every $i\in\N$, $\rho_i \D\M_1$ is a Borel probability measure,
\item[(iii)] for every $i\in\N$, $\rho_i$ is an even piece-wise continuous step function  such that its restriction $\rho_i|_{\R_+}$ to $\R_+$ is decreasing,
\item[(iv)] $\sup\Big\{\frac{\rho_i(x+\varepsilon)}{\rho_i(x)}\,\colon\, x\in\R,\ 0 \Le \varepsilon \Le M_i\Big\}\Le\alpha_i.$
\end{itemize}
(Such $\{M_{i}\}_{i\in\N}$, $\{\alpha_i\}_{i\in\N}$ and $\{\rho_i\}_{i\in\N}$ exist -- see Appendix  \ref{app}.) Consider the measures
\begin{align*}
\mu=\rho_1 \D\M_1\otimes \rho_2 \D\M_1\otimes \ldots
\end{align*}
and
\begin{align*}
\mu_n=\rho_1\D\M_1\otimes\ldots\otimes \rho_n \D\M_1,\quad n\in\N.
\end{align*}
 Suppose $\{p_i\colon\R\to[0,M_i]\ |\ i\in\N\}$ be a family of differentiable functions such that\footnote{This feature of $p_i$'s will be used in Example \ref{exaexa} below} $p_i(x)=0$ for all $x\Le 0$ and $i\in\N$. Let $B\colon\R^\infty\to\R^\infty$ be given by
    \begin{align*}
    B(x_1,x_2,\ldots)=(x_1,x_2+p_2(x_1),x_3+p_3(x_2),\ldots), \quad (x_1,x_2,\ldots)\in\R^\infty,
    \end{align*}
and $A\colon\R^\infty\to\R^\infty$ be the inverse of $B$. For $n\in\N$, let $B_n\colon \R^n\to\R^n$ be given by
    \begin{align*}
    B_n(x_1,x_2,\ldots,x_n)=(x_1,x_2+p_2(x_1),\ldots,x_{n}+p_{n}(x_{n-1}))
    \end{align*}
and let $A_n\colon \R^n\to\R^n$ be its inverse (it exists since the Jacobian determinant of $B_n$ equals $1$). Then, by the change-of-variable theorem (cf. \cite[Theorem 7.26]{rud}), for every $n\in\N$, $C_{A_n}$ is well-defined composition operator in $L^2(\mu_n)$. Moreover, since
    \begin{align*}
    \sup_{x\in\R^n}|\hsf^{A_n}(x)|=\sup_{x\in\R^n}\frac{\rho_1(x_1)\rho_2(x_2+p_2(x_1))\cdots\rho_n(x_n+p_n(x_{n-1}))}{\rho_1(x_1)\cdots\rho_n(x_n)}
    \Le \prod_{i=1}^\infty \alpha_i,
    \end{align*}
we see that $C_{A_n}\in\bb\big(L^2(\mu_n)\big)$. Clearly, the family $\{\hsf^{A_n}\}_{n\in\N}$ is uniformly bounded in $L^\infty$-norms. Hence, in view of Proposition \ref{sufbten}, $C_A\in\bb\big(L^2(\mu)\big)$.
\end{exa}
In the context of gaussian measure $\munfty$ on $\R^\infty$ Proposition \ref{sufbten} reads as follows.
\begin{cor}\label{sufgau}
Let $\A$ be a transformation of $\R^\infty$ induced by a matrix $(a_{ij})_{i,j\in\N}$. Let $\A_n$, $n\in\N$, be the linear transformation of $\R^n$ induced by the matrix $(a_{ij})_{i,j=1}^n$. If the following conditions are satisfied:
\begin{enumerate}
\item[(i)] there exists $\varepsilon>0$ such that $\inf_{n\in\N}|\det \A_n|\geqslant\epsilon$,
\item[(ii)] for every $j\in\N$ there is $K\in\N$ such that $a_{j,k}=0$ for all $k\Ge K$,
\item[(iii)] $\sup_{n\in\N}\|\A_n\|\leqslant1$,
\end{enumerate}
then $C_\A\in \bb\big(L^2(\munfty)\big)$. Moreover, $C_\A$ is the SOT limit of $\{C_{\A_n}\otimes I_n\}_{n\in\N}$, where $I_n$ is the identity operator on $L^2(\munfty)$.
\end{cor}
\begin{proof}
By \cite[Lemma 2.1 and Proposition 2.2]{sto} the operator $C_{\A_n}$ is bounded on $L^2\big(\MUN{n}\big)$
and $\|C_{\A_n}\|^2=|\det \A_n|^{-1}$ for every $n\in\N$. This and \eqref{norm} imply
    \begin{align*}
    \sup_{n\in\N}\|\hsf^{\A_n}\|_{L^\infty(\mu_n)}=\sup_{n\in\N}\|C_{\A_n}\|^2=\bigg(\inf_{n\in\N}|\det \A_n|\bigg)^{-1}<\infty.
    \end{align*}
Since condition (ii) of Corollary \ref{sufgau} yields condition (i) of Proposition \ref{sufbten}, we get the desired conclusion by Proposition \ref{sufbten}.
\end{proof}
The following example of a bounded composition operator in $L^2(\munfty)$ appeared in \cite{mla} and \cite{sto} (it was studied by use of different techniques, not applicable for general matrical symbols).
\begin{exa}
Let $\{a_n\}_{n\in\N}$ be a sequence of real numbers satisfying $0<|a_n|\leq1$ for all $n\in\N$ and $\prod_{n\in\N}|a_n|^{-\frac12}<\infty$.
Let $A\colon\R^\infty\to\R^\infty$ be defined by
    \begin{align*}
    A(x_1,x_2,\ldots)=(a_1x_1,a_2x_2,\ldots),\quad (x_1,x_2,\ldots)\in\R^\infty.
    \end{align*}
By Corollary \ref{sufgau}, $C_\A\in\bb\big(L^2(\munfty)\big)$ and $C_\A=C_{a_1}\otimes C_{a_2}\otimes\dots$.
\end{exa}
Another example of a bounded composition operator in $L^2(\munfty)$ is given below.
\begin{exa}
Let $A\colon\R^\infty\to\R^\infty$ be induced by the matrix $(a_{ij})_{i,j\in\N}$ given by
\begin{align*}
a_{ij}=\left\{%
\begin{array}{ll}
    \exp(-{1/}{i^2}) , & \hbox{for $i=j$;} \\
    \alpha_i, & \hbox{for $j=i+1$;} \\
    0       , & \hbox{otherwise,} \\
\end{array}%
\right.
\end{align*}
where $(\alpha_i)_{i\in\N}$ is a sequence of positive real numbers such that $(a_{ij})_{i,j=1}^n$ is a contraction in $\bb(\R^n)$ for every $n\in\N$. By Corollary \ref{sufgau}, $C_\A\in\bb\big(L^2(\munfty)\big)$.
\end{exa}
As shown below, we can modify the measure $\mu$ from Example \ref{exatriangle1} so that the composition operator induced by the transformation $A$ of $\R^\infty$ as in Example \ref{exatriangle1} is densely defined unbounded (and it does not act in $L^2(\munfty)$).
\begin{exa}\label{exaexa}
Let $\rho_1\D\M_1\otimes \rho_2 \D\M_1 \otimes \ldots$ be the measure defined in Example \ref{exatriangle1}. We will add a hump to density function of every third factor (counted from the second) in the tensor product. For this we consider $\{a_{3i-1}\}_{i\in\N}\subseteq[0,\infty)$, $\{b_{3i-1}\}_{i\in\N}\subseteq(0,\infty)$ and $r>\prod_{i=1}^\infty \alpha_i$ such that
    \begin{itemize}
    \item[(v)] $a_{3i-1}<b_{3i-1}$ for every $i\in\N$,
    \item[(vi)] $b_{3i-1}-a_{3i-1}<M_{3i-1}$ for every $i\in\N$,
    \item[(vii)] $\prod_{i=1}^\infty \big(r^2(b_{3i-1}-a_{3i-1}+2M_{3i-1})+\alpha_{3i-1}^2\big)<\infty$,
    \end{itemize}
where $\{M_{i}\}_{i\in\N}\subseteq(0,\infty)$, $\{\alpha_i\}_{i\in\N}\subseteq(1,\infty)$ are as in Example \ref{exatriangle1} with additional requirement that (v)-(vii) hold (see Appendix  \ref{app}). Let $\{\phi_{3i-1}\colon[a_{3i-1},b_{3i-1}]\to(0,\infty)\ |\ i\in\N\}$ be continuous and satisfy
\begin{itemize}
\item[(viii)] $\sup_{x,y\in[a_{3i-1},b_{3i-1}]}\frac{\phi_{3i-1}(y)}{\phi_{3i-1}(x)}\Ge r$,
\item[(ix)] $\phi_{3i-1}(b_{3i-1})\Le \phi_{3i-1}(x)$ for $x\in[a_{3i-1},b_{3i-1}]$,
\item[(x)] $\int_{[a_{3i-1},b_{3i-1}]}\phi_{3i-1}(x)\M_1(\D x)=\int_{[a_{3i-1},b_{3i-1}]}\rho_{3i-1}(x)\M_1(\D x)$.
\end{itemize}
Define
    \begin{align*}
    \hat\rho_{3i-1}(x)=\left\{%
\begin{array}{ll}
    \phi_{3i-1}(x), & \hbox{$x\in[a_{3i-1},b_{3i-1}]$;} \\
    \rho_{3i-1}(x), & \hbox{$x\notin[a_{3i-1},b_{3i-1}]$.} \\
\end{array}%
\right.
    \end{align*}
Now let $\mu= \eta_1 \D\M_1\otimes \eta_2 \D\M_1\otimes \ldots$, where
\begin{align*}
\eta_n=\left\{%
\begin{array}{ll}
    \hat\rho_n, & \hbox{$n=3k-1$ for $k=1,2,3,\ldots$;} \\
    \rho_n, & \hbox{otherwise.} \\
\end{array}%
\right.
\end{align*}
We may assume that $\eta_i\Le 1$ for $i\in\N$ (see Appendix  \ref{app}). If $\{A_n\}_{n\in\N}$ are as in Example \ref{exatriangle1}, then the family of Radon-Nikodym derivatives $\{\hsf^{A_n}\}_{n\in\N}$, where $\hsf^{A_n}=\frac{\D\mu_n\circ A_n^{-1}}{\D\mu_n}$ with $\mu_n=\eta_1\D\M_1\otimes\ldots\otimes\eta_n\D\M_1$, satisfies $\sup_{n\in\N}\|\hsf^{A_n}\|_{L^2(\mu_n)}<\infty$. Indeed, first fix $k\in\N$ and take $n=3k-1$. Then
    \begin{align*}
    \int\limits_{\R^n}&(\hsf^{A_n})^2\D\mu_n
    =\int\limits_{\R^n}\left(\frac{\eta_1(x_1)\eta_2(x_2+p_2(x_1))\cdots \eta_n(x_n+p_{n}(x_{n-1}))}{\eta_1(x_1)\eta_2(x_2)\cdots \eta_n(x_n)}\right)^2\D\mu_n\\
    &\Le\Big(\prod_{i=1}^n\alpha_i^2\Big)
    \int\limits_{\R^n}\left(\frac{\prod_{i=1}^k\hat\rho_i(x_{3i-1}+p_{3i-1}(x_{3i-2}))}{\prod_{i=1}^k\hat\rho_i(x_{3i-1})}\right)^2\prod_{i=1}^n\eta_i(x_{i})\D\M_n\\
    &\Le C\prod_{i=1}^k\int\limits_{\R^2}\left(\frac{\eta_{3i-1}(x_{3i-1}+p_{3i-1}(x_{3i-2}))}{\eta_{3i-1}(x_{3i-1})}\right)^2\eta_{3i-1}(x_{3i-1})\eta_{3i-2}(x_{3i-2})\D\M_2.
    \end{align*}
For $i\in\{1,\ldots,k\}$, we can divide $\R^2$ into two disjoint sets $\Omega_i$ and $\R^2\setminus\Omega_i$ where
\begin{align*}
\Omega_i=\{(y,x)\in\R^2| x+p_{3i-1}(y)\in[a_{3i-1},b_{3i-1}]\, \vee\, x\in[a_{3i-1},b_{3i-1}]\}.
\end{align*}
Then the $\eta_{3i-1}\otimes\eta_{3i-2}\D\M_2$ measure of $\Omega_i$ is less or equal to $b_{3i-1}-a_{3i-1}+2M_{3i-1}$ and
    \begin{align*}
    \frac{\eta_{3i-1}(x+p_{3i-1}(y))}{\eta_{3i-1}(x)}\Le r,\quad (y,x)\in\Omega_i.
    \end{align*}
Hence we obtain
    \begin{multline*}
    \int\limits_{\R^2}\left(\frac{\eta_{3i-1}(x_{3i-1}+p_{3i-1}(x_{3i-2}))}{\eta_{3i-1}(x_{3i-1})}\right)^2\eta_{3i-1}(x_{3i-1})\eta_{3i-2}(x_{3i-2})\D\M_2\\
    \Le r^2(b_{3i-1}-a_{3i-1}+2M_{3i-1})+\alpha_{3i-1}^2,\quad i=1,\ldots,k.
    \end{multline*}
Consequently, we get
    \begin{align*}
    \int_{\R^{3k-1}}|\hsf^{A_{3k-1}}|^2d\mu_{3k-1}\Le C \prod_{i=0}^k
    \big(r^2(b_{3i-1}-a_{3i-1}+2M_{3i-1})+\alpha_{3i-1}^2\big).
    \end{align*}
This means that the Radon-Nikodym derivatives $\{\hsf^{A_n}\}_{n\in\N}$ are uniformly bounded in $L^2$ as desired. By Theorem \ref{absconA}, the operator $C_A$, where $A$ is as in Example \ref{exatriangle1}, is densely defined and closed.

Now we prove that $C_A$ is unbounded. This follows from the fact that the Radon-Nikodym derivatives $\{\hsf^{A_n}\}_{n\in\N}$ are not uniformly bounded in $L^\infty$ norm. Indeed, since the image of the function $\R\times [a_{k},b_{k}]\ni(x, y)\to y+p_k(x)\in \R$ contains the interval $[a_k,b_k]$ for every $k=3i-1$, $i\in\Z_+$, we see that for every $i\in\N$ there are $(\hat x_{3i-2}, \hat x_{3i-1})\in\R\times [a_{3i-i},b_{3i-1}]$ such that
\begin{align*}
\frac{\eta_{3i-1}(\hat x_{3i-1}+p_{3i-1}(\hat x_{3i-2}))}{\eta_{3i-1}(\hat x_{3i-1})}\Ge r,\quad i\in\N.
\end{align*}
This and properties of $\rho_i$'s and $p_i$'s imply that for every fixed $x_{3i+1}\in\R$ we have
\begin{align*}
\frac{\eta_{3i-1}(\hat x_{3i-1}+p_{3i-1}(\hat x_{3i-2}))\eta_{3i}(x_{3i}+p_{3i}(x_{3i-1}))\eta_{3i+1}(x_{3i+1}+p_{3i+1}(x_{3i}))}{\eta_{3i-1}(\hat x_{3i-1})\eta_{3i}(x_{3i})\eta_{3i+1}(x_{3i+1})}\Ge r
\end{align*}
for some sufficiently small $x_{3i}<0$, for every $i\in\N$. As a consequence, we have
\begin{align*}
&\esup_{x\in\R^{3i+1}} \hsf^{A_{3i+1}}(x)
=\esup_{x\in\R^{3i+1}} \frac{(\eta_1\otimes\ldots\otimes\eta_{3i+1})\circ A_{3i+1}^{-1}(x)}{(\eta_1\otimes\ldots\otimes\eta_{3i+1})(x)}\\
&=\esup_{x\in\R^{3i+1}}\frac{\eta_1(x_1)\eta_2(x_2+p_2(x_1))\cdots \eta_n(x_{3i+1}+p_{{3i+1}}(x_{{3i}}))}{\eta_1(x_1)\eta_2(x_2)\cdots \eta_{3i+1}(x_{3i+1})}
\Ge r^i,\quad i\in\N
\end{align*}
which proves our claim.
\end{exa}
\appendix
\section{}\label{app}
In this appendix we provide $\{\alpha_i\}_{i\in\N}$, $\{M_{i}\}_{i\in\N}$, $\{\rho_i\}_{i\in\N}$,  $\{a_{3i-1}\}_{i\in\N}$, $\{b_{3i-1}\}_{i\in\N}$ and $\{\phi_{3i-1}\}_{i\in\N}$ satisfying conditions (i)-(x) Of Examples \ref{exatriangle1} and \ref{exaexa}.

Let us begin with a decreasing sequence $\{\alpha_i\}_{i\in\N}\subseteq(1,\infty)$ such that $\prod_{i\in\N}\alpha_i<\infty$. Choose $\{\beta_{3i-1}\}_{i\in\N}\subseteq(0,\infty)$ so that $\prod_{i\in\N}(\beta_{3i-1}+\alpha_{3i-1}^2)<\infty$. Now, to obtain all objects required in Example \ref{exatriangle1} we proceed in steps.

{\bf Step 1.} For $i\in\N$, we set $M_i=2^{-1}(1-\alpha_i^{-1})$ and $\rho_i(x)=\alpha_i^{-n}$ for $x\in \R$ with $nM_i\Le |x|<(n+1)M_i$, $n\in\Z_+$.

It is elementary to verify that $\{M_{i}\}_{i\in\N}$ and $\{\rho_i\}_{i\in\N}$ as in Step 1 satisfy requirements (i)-(iv) of Example \ref{exatriangle1}. Observe that $\{M_{i}\}_{i\in\N}$ is decreasing. Let $r\in\R$ satisfy $r>\prod_{i\in\N}\alpha_i$. Fix $i\in\N$.

{\bf Step 2.} If $3 r^2 M_{3i-1}>\beta_{3i-1}$, then we choose $k\in\N$ such that $3 k^{-1}r^2M_{3i-1}<\beta_{3i-1}$ and substitute $M_n$ by $M_nk^{-1}$ for all $n\Ge i$, leaving $\rho_n$'s as they were. If $3r^2 M_{3i-1}<\beta_{3i-1}$, then we skip any substitutions and go directly to the next step.

{\bf Step 3.} Take $a_{3i-1}=0$ and any $b_{3i-1}\in (0,M_{3i-1})$. Let $\phi_{3i-1}\colon [a_{3i-1},b_{3i-1}]\to(0,\infty)$ be any continuous function such that $\phi_{3i-1}(0)=1$, and $\phi_{3i-1}(b_{3i-1})=r$, and (x) of Example \ref{exaexa} holds.

By elementary calculations we verify that $\{\alpha_i\}_{i\in\N}$, $\{\hat\rho_i\}_{i\in\N}$, $\{M_{i}\}_{i\in\N}$, $\{a_{3i-1}\}_{i\in\N}$, $\{b_{3i-1}\}_{i\in\N}$ and $\{\phi_{3i-1}\}_{i\in\N}$ satisfy conditions (i)-(x) of Examples \ref{exatriangle1} and \ref{exaexa}.

It is possible that for some $i\in\N$, condition $\eta_{3i-1}\Le1$ does not hold (this is possible if $\sup_{x\in[a_{3i-1}, b_{3i-1}]}\phi_{3i-1}>1$). If this is the case, then another slight modification is needed. This may be done in various way. Below we propose a one based on the construction above.

{\bf Step 4.} If $\delta:=\sup_{x\in[a_{3i-1}, b_{3i-1}]}\phi_{3i-1}>1$, then there exists $x_0>0$ such that $1=x_0+\delta^{-1}\int_{\R}\eta_{3i-1}(t)\M_1(\D t)$.
The modified $\eta_{3i-1}$, which satisfies all the requirements, has the form
\begin{align*}
\eta_{3i-1}(t)= \left\{
                   \begin{array}{ll}
                     \delta^{-1}\eta_{3i-1}(t), & \hbox{for $t\Le0$;} \\
                     1, & \hbox{for $t\in[0,x_0]$;} \\
                     \delta^{-1}\eta_{3i-1}(t-x_0), & \hbox{for $t\Ge x_0$.}
                   \end{array}
                 \right.
\end{align*}

\end{document}